\documentclass[11pt,leqno]{amsart}
\usepackage{amsmath}
\usepackage{amsfonts}
\usepackage{amssymb}
\usepackage{graphicx}
\usepackage{hyperref}
\usepackage{enumerate}
\usepackage{xcolor}
\newtheorem{theorem}{Theorem}
\newtheorem{theoremA}{Theorem}

\newtheorem{corollaryA}[theoremA]{Corollary}
\newtheorem{lemmaA}[theoremA]{Lemma}

\newtheorem{corollary}[theorem]{Corollary}

\newtheorem{definition}[theorem]{Definition}

\newtheorem{lemma}[theorem]{Lemma}

\newtheorem{problem}{Problem}

\newtheorem{remark}[theorem]{Remark}

\newcommand{\inte}{\mathrm{int}}

\newcommand{\rr}{\mathbb{R}}

\newcommand{\zz}{\mathbb{Z}}
\newcommand{\nn}{\mathbb{N}}
\newcommand{\sect}{\operatorname{Sect}}
\newcommand{\ric}{\operatorname{Ric}}

\newcommand{\bb}{\mathbb{B}}

\newcommand{\scal}{\operatorname{Scal}}
\newcommand{\curv}{\mathrm{Curv}}

\newcommand{\R}{\mathbb{R}}
\newcommand{\N}{\mathbb{N}}

\newcommand{\pM}{{\partial M}}
\renewcommand{\a}{{\alpha}}
\renewcommand{\b}{{\beta}}
\newcommand{\g}{{\gamma}}
\newcommand{\e}{{\epsilon}}
\newcommand{\s}{{\sigma}}

\renewcommand{\l}{{\lambda}}

\begin{document}

\title[Complete Riemannian extension]{The smooth Riemannian extension problem: completeness}
\author{Stefano Pigola}
\address{Universit\`a dell'Insubria, Dipartimento di Scienza e Alta Tecnologia\\
Via Valleggio 11, 22100 Como, Italy} \email{stefano.pigola@uninsubria.it}
\author{Giona Veronelli}
\address{Universit\'e Paris 13, Sorbonne Paris Cit\'e, LAGA, CNRS ( UMR 7539)
99\\
avenue Jean-Baptiste Cl\'ement F-93430 Villetaneuse - FRANCE} \email{veronelli@math.univ-paris13.fr}
\date{December 21, 2015}

\begin{abstract}
By means of a  general gluing and conformal-deformation construction, we prove that any smooth, metrically complete Riemannian manifold with smooth boundary can be realized as a closed domain into a smooth, geodesically complete Riemannan manifold without boundary. Applications to Sobolev spaces, Nash embedding and local extensions with strict curvature bounds are presented.

\end{abstract}
\maketitle
\tableofcontents

\section*{Introduction}

Let $(M,g_M)$ be a given Riemannian manifold with smooth (possibly non-compact) boundary $\partial M \not = \emptyset$ and subjected to some constraint on one of its Riemannian invariants, such as a curvature (or a volume growth) bound. The general problem we are interested in consists in understanding when, and to what extent, the original manifold $M$ can be prolonged past its boundary in order to obtain a new smooth Riemannian manifold $(\bar M ,  g_{\bar M})$, this time without boundary, such that one of the invariants alluded to above is kept controlled. Clearly, the most interesting situation occurs when the extended metric can be taken to be geodesically complete. In this case we can speak of $(\bar M,  g_{\bar M})$ as a \textit{complete Riemannian extension of $(M,g_M)$ with controlled Riemannian invariants}. First insights into the possibility of constructing a complete prolongation were given by S. Alexander and R. Bishop in \cite {AB}. Actually, this paper is mostly focused on the prolongation of open manifolds without boundary, but it contains useful information also in the boundary case. The smooth extension, via gluing techniques, of compact manifolds with a strict Ricci curvature lower bound and a convexity condition on the boundary was considered by G. Perelman in \cite{Per}; see also \cite{Wan}. Extensions of compact manifolds with non-negative scalar curvature up to the mean convex boundary are contained in \cite{R}. The extended metric is just $C^{2}$ but this is (abundantly) enough to get interesting rigidity results based on the positive mass theorem. Very recently, \cite{AMW}, a gluing technique in the spirit of \cite{Per} has been applied to prove that the space of metrics with non-negative Ricci curvature and convex boundary on the Euclidean three ball is path connected. In a somewhat different direction, gluing methods have been also employed by J. Wong, \cite{Won} in order to obtain isometric extensions with totally geodesic boundary and a metric-curvature lower bound in the sense of Alexandrov. This has applications to Gromov-Hausdorff precompactness results and volume growth estimates.\smallskip

In view of the well known relations between the topology of a complete Riemannian manifold and the bounds on its curvatures, or its volume growth, we are naturally led to guess that some topological obstruction appears somewhere in the extension process. In this direction, it would be important to verify whether some of these obstructions are encoded in the original piece with boundary and this requires, first, a phenomenological investigation over concrete examples. For instance, a complete extension with non-negative Ricci curvature should be forbidden in general. In this respect note that the topology of a compact manifold with convex boundary and positive Ricci curvature cannot be too much wide and this is compatible with the positive results we have mentioned above. For instance, by Bochner methods, the first real Betti number must vanish and, thus, the fundamental group cannot contain $\zz^{k}$ as a free factor. Topological obstructions should also appear at the level of upper sectional curvature bounds. Think for instance to the Cartan-Hadamard theorem, valid in the setting of geodesic metric spaces. The possibility of extending a complete simply connected manifold with boundary and negative curvature $K<0$ to a complete manifold with sectional curvature controlled by $K+\epsilon$ was addressed by S. Alexander, D. Berg and R. Bishop, \cite[p. 705]{ABB-TAMS}, during their investigations on isoperimetric properties under the assumption that the boundary has negative curvature on its concave sections. We are grateful to S. Alexander for pointing out this reference. In sharp contrast, in view of J. Lohkamp insights, \cite{Loh-Annals2}, it is expected that an upper Ricci curvature bound imposes no restrictions at all. The related question concerning how many Riemannian extensions could be obtained is interesting in its own. For instance, in the context of Einstein manifolds, because of the real analyticity of the Einstein metric, it is reasonable that some strong metric rigidity arises once the topological information (from the fundamental group viewpoint) is concentrated in the original part.\smallskip

This very brief and informal discussion serves to outline a major project concerning the systematic investigation around the  \textit{Riemannian extension problem}.
\begin{definition}
 Let $(M,g_{M})$ be a smooth $m$-dimensional Riemannian manifold with possibly nonempty boundary. A Riemannian extension of $(M,g_{M})$ is any smooth $m$-dimensional Riemannian manifold $(N,g_{N})$ with possibly non-empty boundary such that $M$ is isometrically embedded in $\inte N$.
\end{definition}
Roughly speaking, this project could be articulated in the following problems that represent the basic steps towards a suitable understanding of the subject.

\begin{problem}[completeness]\label{problem-completeness}
Let $(M,g_M)$ be a metrically complete Riemannian manifold with smooth boundary $\partial M \not=\emptyset$. Does there exist a geodesically complete Riemannian extension $(\bar M,g_{\bar M})$ of $M$ with $\partial \bar M = \emptyset$?
\end{problem}

\begin{problem}[curvature constraints]\label{problem-curvature}
 Let $\curv$ denote either of the curvatures $\sect$, $\ric$ or $\scal$ and let $C\in\R$. Let $(M,g_M)$ be a smooth $m$-dimensional, (non-nencessarily complete) Riemannian manifold with smooth boundary $\partial M \not=\emptyset$ satisfying $\curv_{g_M} < C$ (resp. $\leq C$, $>C$ or $\geq C$). Does there exist a complete, $m$-dimensional Riemannian extension $(\bar M ,g_{\bar M})$ with $\partial \bar M = \emptyset$ and such that the same curvature constraint holds?
\end{problem}

\begin{problem}[rigidity]\label{problem-rigidity}
 Let $(M,g_M)$ be a smooth, $m$-dimensional, complete Riemannian manifold with smooth boundary $\partial M \not=\emptyset$ satisfying $\ric = \lambda \in \rr$. Which topological conditions on $M$ are required to guarantee that a complete Einstein extension $(\bar M , g_{\bar M})$ exists and is essentially unique?
\end{problem}

We shall deal with Problems \ref{problem-curvature} and \ref{problem-rigidity} in the forthcoming paper \cite{PV-curv} where, in particular: we will examine, under different viewpoints,  a number of examples showing obstructions to the existence of  a complete Riemannian extension under Ricci curvature lower bounds and sectional curvature upper bounds; we will obtain positive results for upper Ricci curvature bounds; and finally,  we will investigate rigidity issues in the Einstein case.

\medskip

\noindent The aim of the present paper is to give an exhaustive answer to Problem  \ref{problem-completeness}.
\begin{theoremA}\label{theorem-extension}
Let $(M,g_{M})$ be an $m$-dimensional connected Riemannian manifold with  smooth boundary $\pM \not= \emptyset$. Let $Q$ be any smooth $m$-dimensional differentiable manifold whose nonempty boundary $\partial Q$ is diffeomorphic to $\pM$. Then, there exists a Riemannian extension $(N,g_{N})$ of $(M,g_{M})$ such that $N\setminus M$ is diffeomorphic to the interior of $Q$. Moreover, if $(M,g_{M})$ is complete, then the extension $(N,g_{N})$ can be constructed to be complete.
\end{theoremA}
In particular, by choosing $Q=M$ in the previous statement, with the trivial identification of the boundaries, we get
\begin{corollaryA}\label{th_extension}
Let $(M,g_M)$ be a smooth complete, $m$-dimensional Riemannian manifold with smooth nonempty boundary $\partial M$. Then, there exists a geodesically complete Riemannian extension $(N,g_{N})$ of $(M,g_{M})$ with $\partial N = \emptyset$. 
\end{corollaryA}
These results are then applied in several directions. First, we show that Problem \ref{problem-curvature} has always a positive answer (regardless of any restriction on the boundary) if we remove the completeness assumption and we require strict curvature bounds; see Corollary \ref{corollary-localcurvature}. Next, we prove a density result \`a la Meyers-Serrin concerning first order Sobolev spaces on complete manifolds with boundary; see Corollary \ref{corollary-density}. Finally, as a direct consequence of Nash theorem, we observe that a complete Riemannian manifold with boundary has a proper isometric embedding into a Euclidean space; see Corollary \ref{corollary-nash}.

\section{The general gluing-deformation construction}
In this section we prove Theorem \ref{theorem-extension}. The manifolds $M$ and $Q$ are glued along the diffeomorphic boundaries and, using this ambient space, the original metric of $M$ is readily extended. At this point, the complete Riemannian extension is obtained via a careful conformal deformation. The proof that the deformed metric is actually complete relies on metric-space arguments.

\subsection{Preliminaries on metric spaces}

Given a metric space $(X,d)$, a continuous path $\gamma : [a,b] \to X$ is rectifiable if
\[
L_{d}(\gamma) := \sup \sum_{i=1}^{n} d(\gamma(t_{i-1}), \gamma(t_{i})) < +\infty
\]
where the supremum is taken with respect to all the partitions $t_{0}=a < t_{1} < \cdots < t_{n}=b$ of the interval $[a,b]$. In this case, the number $L_{d}(\gamma)$ is the metric-length of $\gamma$ and it is invariant by reparametrizations of the curve. On the metric space $(X,d)$ it is defined a length-distance given by
\[
d_{L} (x ,y) = \inf L_{d}(\gamma)
\]
the infimum being taken with respect to all rectifiable paths (if any) connecting $x$ to $y$. Observe that Lipschitz paths are trivially rectifiable and, conversely, every rectifiable path can be reparametrized to a constant speed, hence Lipschitz, path \cite[Proposition 2.5.9]{BBI}.
The metric space $(X,d)$ is a length metric space if $d = d_{L}$.\smallskip

Let $(M,g_M)$ be a smooth Riemannian manifold with (possibly empty) boundary $\partial M$. Its intrinsic distance, which is defined as the infimum of  the Riemannian lengths of piecewise $C^{1}$ paths connecting two given points, is denoted by $d_{(M,g_M)}$. It is well known that the metric space $(M,d_{(M,g_M)})$ is a length metric space. The Riemannian manifold $(M,g_M)$ is said to be complete if $(M,d_{(M,g_M)})$ is a complete metric space. Since $M$ is locally compact, the length-metric version of the Hopf-Rinow theorem, \cite[Theorem 2.5.28]{BBI} and Theorem \ref{th_HR} below, implies that the metric completeness of $M$ is equivalent to the Heine-Borel property which, in turn, is equivalent to the fact that a geodesic path $\gamma: [a,b) \to M$ extends continuously to the endpoint $b$. Here, by a geodesic, we mean a Lipschitz path which minimizes the distance between any couple of its points. It is well known that it is $C^{1}$ regular, \cite{AA, ABB-Illinois}.\smallskip

A further notion of completeness that turns out to be very useful in applications involves the length of divergent paths. This characterization will be used to show that the glued manifold constructed in the next section is complete.

\begin{definition}
 Let $(X,d)$ be a metric space (e.g. a Riemannian manifold with possibily non-empty boundary with its intrinsic metric). A continuous path $\gamma : [a, b) \to X$ is said to be a divergent path if, for every compact set $K \subset M$, there exists $a\leq T <b$ such that $\gamma(t) \not \in K$ for every $T \leq t < b$.  The metric space $(X,d)$ is called ``divergent paths complete'' (or complete with respect to divergent paths) if every locally Lipschitz divergent path $\gamma : [0,1) \to X$ has infinite length where, clearly, $L_{d}(\gamma) = \lim_{\delta \to 1} L_{d}(\gamma|_{[0,\delta]})$.
\end{definition}

It is well known that for a manifold without boundary, the notions of metric (hence geodesic) completeness and of divergent paths completeness are equivalent.
Let us point out that a similar equivalence holds more generally on a locally compact length space
hence, in particular, for manifolds with smooth boundaries. Namely, we have the following

\begin{theorem}[Hopf-Rinow]\label{th_HR}
Let $(X,d)$ be a locally compact length space. The following assertions are equivalent.
\begin{enumerate}
\item $(X,d)$ is metrically complete, i.e. it is complete as a metric space.
\item $(X,d)$ satisfies the Heine-Borel property, i.e. every closed metric ball in $X$ is compact.
\item $(X,d)$ is geodesically complete, i.e. every constant speed geodesic $\gamma:[0,a)\to X$ can be extended to a continuous path $\bar\gamma:[0,a]\to X$
\item Every Lipschitz path $\gamma:[0,a)\to X$ can be extended to a continuous path $\bar\gamma:[0,a]\to X$
\item $(X,d)$ is divergent paths complete, i.e. every locally Lipschitz divergent path $\gamma:[0,a)\to X$ has infinite length.
\end{enumerate}
\end{theorem}

\begin{proof}
It is proven in \cite{BBI} that $(1)\Leftrightarrow(2)\Leftrightarrow(3)$. Moreover, $(4)\Rightarrow (3)$ trivially. \\
We prove that $(2)\Rightarrow (5)$. For $n\in\N$, consider the compact sets $B_n^X(\gamma(0))$. Since $\gamma$ is divergent, there exists a sequence  $\{t_n\}_{n=1}^\infty\subset (0,a)$ such that $\gamma(t_n)\not\in \overline{B}_n^X(\gamma(0))$. In particular
$$
L_d(\gamma)\geq L_d(\gamma|_{[0,t_n]})\geq d(\gamma(0),\gamma(t_n))\geq n.
$$
Since $n$ can be arbitrarily large, $\gamma$ has infinite length.\\
To conclude, we prove that $(5)\Rightarrow (4)$. Let $\gamma:[0,a)\to X$ be a Lipschitz rectifiable path. Since $\gamma$ is defined on $[0,a)$ and is Lipschitz, it has finite length. Then it can not be divergent. Namely, there exists a compact set $K\subset X$ and a sequence $\{t_n\}_{n=1}^\infty\subset (0,a)$ such that $t_n\to a$ as $n\to\infty$ and $\gamma(t_n)\in K$ for all $n$. By compactness of $K$, up to passing to a subsequence, $\gamma(t_n)\to x$ as $n\to\infty$ for some limit point $x\in K$. Set $\gamma(a)=x$. We are going to show that $\gamma:[0,a]\to X$ is continuous. Fix $\e>0$. Take $N\in\N$ large enough such that $d(\gamma(t_n),x)<\e/2$ for all $n\geq N$ and $t_N>a-\frac{\e}{2C_\g}$, where $C_\g$ is the Lipschitz constant of $\g$. Then for all $t\in(t_N,a)$,
$$
d(\gamma(t),x)\leq d(\gamma(t),\g(t_N))+d(\g(t_N),x)\leq C_\g|t-t_N|+\e/2 \leq \e.
$$
\end{proof}

We shall need to consider metric properties of curves into a manifold with boundary with respect to both the original metric and to the extended one. To this end, the following Lemma will be crucial.
\begin{lemma}\label{lemma-length}
Let $(N,g_{N})$ be a Riemannian extension of the manifold with boundary $(M,g_{M})$ and let $\gamma : [0,1] \to M$ be a fixed curve. Then
\begin{enumerate}
\item[(a)] $\gamma$ is $d_{(N,g_{N})}$-Lipschitz (resp. rectifiable) if and only if it is $d_{(M,g_{M})}$-Lipschitz (resp. rectifiable).
\end{enumerate}
Moreover, in this case:
\begin{enumerate}
 \item [(b)] $L_{g_{M}}(\gamma) = L_{g_{N}}(\gamma)$.
 \item [(c)] The speed $v_{\gamma}$ of $\gamma$, in the sense of \cite{BBI}, is the same when computed with respect to $d_{(M,g_{M})}$ and $d_{(N,g_{N})}$.
\end{enumerate}
\end{lemma}
\begin{proof}
We preliminarily observe that $d_{(N,g_{N})} \leq d_{(M,g_{M})}$ on $M$. \smallskip

(a) It is enough to consider the Lipschitz property because, as we have already recalled, every rectifiable path has a Lipschitz (constant speed) reparametrization.

We assume that $\gamma$ is $d_{(N,g_{N})}$-Lipschitz  and we prove that $\gamma$ is $d_{(M,g_{M})}$-Lipschitz, the other implication being trivial from the above observation. We shall show that, for every $t_{0} \in [0,1]$, there exists a closed interval $I_{0} \subset [0,1] $ containing $t_{0}$ in its interior such that $\gamma|_{I_{0}}$ is $d_{(M,g_{M})}$-Lipschitz.  We suppose that $\gamma(t_{0}) \in \partial M$, the other case being easier. Let $\varphi_{0} : U_{0} \to \bb_{1}$ be a local coordinate charts of $N$ centered at $\gamma(t_{0})$ and such that $\varphi_{0}(U_{0} \cap M) = \bb_{1}^{+}$, the upper-half unit ball. Let $V_{0} = \varphi_{0}^{-1}(\bb_{1/2})$ and choose $I_{0}$ such that $\gamma(I_{0}) \subset V_{0}$. Note  that the distances $d_{(N,g_{N})} $ and $d_{(V_{0} ,g_{N})}$ are equivalent on $V_{0}$ and, similarly, $d_{(M,g_{M})} $ and $d_{(V_{0}\cap M ,g_{M})}$ are equivalent on $V_{0} \cap M$. Moreover, $\varphi_{0} : (V_{0},d_{(V_{0},g_{N})}) \to (\bb_{1/2},d_{(\bb_{1/2},g_{\mathrm{Eucl}})})$ and $\varphi_0 : (V_{0} \cap M ,d_{(V_{0} \cap M,g_{M})}) \to (\bb^{+}_{1/2} ,d_{(\bb^{+}_{1/2} ,g_{\mathrm{Eucl}})})$ are  bi-Lipschitz.
Since $\gamma$ is $d_{(N,g_{N})}$-Lipschitz then $\varphi_{0} \circ \gamma|_{I_{0}}$ is $d_{(\bb_{1/2},g_{\mathrm{Eucl}})}$-Lipschitz. Since $\bb^{+}_{1/2}$ is convex then $\varphi_{0} \circ \gamma|_{I_{0}}$ is $d_{(\bb^{+}_{1/2},g_{\mathrm{Eucl}})}$-Lipschitz. Hence $\g|_{I_{0}}$ is $d_{(M,g_{M})}$-Lipschitz.\smallskip

(b) Using a partition of $[0,1]$ by sufficiently small subintervals we can apply \cite[Lemma 1 and Lemma 3]{AA}.\smallskip

(c) This follows from (b) and \cite[Corollary 2.7.5]{BBI}.
\end{proof}

\subsection{The proof of Theorem \ref{theorem-extension}}
Let $g_{Q}$ be any  Riemannian metric on $Q$ and let $\eta: \partial M \to \partial Q$ be a selected diffeomorphism. Let us consider the smooth gluing $N:= M \cup_{\eta} Q$ whose differentiable structure is obtained in a standard way using collar neighborhoods of the manifolds involved. More precisely, $N$ is the topological manifold without boundary obtained from $M \cup Q$ identifying points $x$ and $\eta(x)$ for every $x \in \partial M$. 
With a slight abuse of notation, here and on we consider $M$ and $Q$ as subsets of $N$ such that $M\cap Q=\pM$, and we identify objects on $M$ and $Q$ with their images on $N$ via the inclusions $M\hookrightarrow N$ and $Q\hookrightarrow N$.
Let $\mathcal W_M\subset M$ be an open tubular neighborhood of $\pM$ and let $p_M:\mathcal W_M\to\pM\times (-1,0]$ 
be the corresponding smooth diffeomorphism, whose restriction 
$p_M|_\pM : \pM\subset \mathcal W_M \to \pM\times 0$ is the identity map $p_M (x) = x \times 0$. 
Similarly, let $\mathcal W_Q\subset Q$ be a  tubular neighborhood of $\partial Q$ and let $p_Q: \mathcal W_Q\to\pM\times [0,1)$ 
be the corresponding smooth diffeomorphism, whose restriction 
$p_Q|_\pM : \pM\subset W_Q \to \pM\times 0$ is the identity map $p_Q (x) = x \times 0$.

Then $p_M$ and $p_Q$ induce a homeomorphism $p:\mathcal W= \mathcal W_M\cup \mathcal W_Q\subset N \to \pM\times (-1,1)$. The differentiable structure on $N$ is obtained by imposing that the homeomorphism $p$ is a smooth diffeomorphism and that the inclusions 
$j_M: \pM \hookrightarrow N$ and $j_Q: \partial Q \hookrightarrow N$ are smooth embeddings.\smallskip

The proof of Theorem \ref{theorem-extension} is now achieved in three steps that we formulate as the following Lemmas of independent interest.

\begin{lemmaA}\label{lemma-extension-1}
 Keeping the above notation, there exists a Riemannian metric $\tilde g$ on $N$ such that $\tilde g = g_{M}$ on $M$, i.e., $(N,\tilde g)$ is a Riemannian extension of $(M,g_{M})$. Moreover,  for every $\epsilon >0$, there exists a tubular neighborhood $\mathcal X_{Q}\subseteq \mathcal{W}_{Q}$ of $\pM$ in $N\setminus M$ such that:
\begin{enumerate}
 \item [(a)] $P := M \cup \mathcal{X}_{Q} \subset N$ is a manifold with smooth boundary.
 \item [(b)] there exists a $(1+\epsilon)$-Lipschitz projection $\rho:  (P,\tilde g) \to (M,g_M)$ such that $\rho|_{\mathcal X_{Q}}$ is a diffeomorphism.
\end{enumerate}
\end{lemmaA}
In what follows, the value of $\epsilon$ is irrelevant. Therefore, we will always assume that $\epsilon = 1$.
\begin{lemmaA}\label {lemma-extension-2}
 Let $(M,g_{M})$ and $(P , g_{P}= \tilde g|_{P})$ be as above. If $(M,g_{M})$ is complete then  so is $(P,g_{P})$.
\end{lemmaA}

\begin{lemmaA}\label{lemma-extension-3}
Let $(M,g_{M})$ be an $m$-dimensional  Riemannian manifold with non-empty boundary. Let $(P,g_{P})$ be a complete Riemannian extension of $M$ with non-empty boundary. Let $(N,\tilde g)$ be a Riemannian extension of $(P,g_{P})$, hence of $(M,g_{M})$. Then, there exists a Riemannian metric $g_{N}$ on $N$ such that $(N,g_{N})$ is still a Riemannian extension of $(M,g_{M})$ and it is complete.
\end{lemmaA}

The rest of the section is entirely devoted to the proofs of these results.

\begin{proof}[Proof of Lemma \ref {lemma-extension-1}]
 We proceed by steps.\smallskip

 \noindent \textit{Step 1.} First, we construct a local extension of $g_{M}$ beyond $\pM$ in $N$.
Consider on the cylinder $\pM\times(-1,1)$ a locally finite family of coordinate charts $\{(V_\b,\psi_\b): \b\in B\}$ such that 
\begin{itemize}
 \item [(i)] $\cup_{\b\in B}V_\b\supset \pM\times \{0\}$,
 \item [(ii)] $\psi_\b(V_\b)= \bb_1$,
\end{itemize}
where $\bb_{1}$ denotes the unit ball in the Euclidean space $\rr^{m}$. Let $\mathcal S$ be the space of symmetric $m\times m$ matrices and set 
$$\mathbb{L}_{t}=\{(x_1,\dots,x_n)\in \bb_{1} : x_n\leq t\}.$$ 
In particular, $\mathbb{L}_{0} = \bb_{1}^{-}$, the lower-half unit ball. Fix $\b\in B$. The metric $g_{M}$ on $p^{-1}(V_\b)
\cap M$ is represented in local coordinates by a smooth section 
$s_\b:\mathbb{L}_0\to\mathcal S$, such that $s_\b(x)$ is positive definite for all $x\in \mathbb{L}_0$. Extend smoothly $s_\b$ to a section $\tilde s_\b: \bb_{1}  \to\mathcal S$. 
By continuity we can find a $t_\b\in (0,1]$ such that  $\tilde s_\b$ is positive definite for all $x\in \mathbb{L}_{t_\b}$. Define $\tilde V_\b=p^{-1}\circ\psi_\b^{-1}(\operatorname{int}\mathbb{L}_{t_\b})$. Repeating the construction for all $\b\in B$ we have obtained a family of local Riemannian metrics $\tilde g_\b$ defined on $\tilde V_\b$ for all $\b\in B$, such that $\tilde g_\b=g_{M}$ on $\tilde V_\b\cap M$. Moreover $\cup_{\b\in B}\tilde V_\b\supset \pM\times \{0\}$.\\

\noindent \textit{Step 2.} Next, we extend smoothly $g_{M}$ to a global metric $\tilde g$ on $N$. 
The collection of sets $\{\inte M, \inte Q,  \tilde V_\b : \b\in B\}$
gives a locally finite covering of $N$. Let $\{\eta_M,\eta_Q, \eta_\beta : \beta\in B\}$ be a subordinated partition of unity. Then 
$$\tilde g=\eta_M g_M + \eta_Q g_Q  + \sum_{\b\in B}\eta_\b\tilde g_\b$$ is a positive definite smooth Riemannian metric on $N$. Moreover, for all $x\in M$,
$$\tilde g|_x=\eta_M g_M|_x + \sum_{\b\in B}\eta_\b g|_x=g_M|_x.$$

\noindent \textit{Step 3.} Finally, we show how to construct the neighborhood $\mathcal X_Q$ and the Lipschitz projection $\rho$. 

For all $x\in\pM$, let $\nu(x)$ be the outward normal vector to $\pM$ at the point $x$. The exponential map
$\exp^\perp(x,s):=\exp_x(s\nu(x))$ is well defined for any $s$ small enough (depending on $x$), i.e. for $s\in [-s_0(x),s_0(x)]$ where we can assume that $s_0:\pM\to(0,\infty)$ is smooth. Set 
\begin{align*}
\mathcal X_Q&= \{\exp^\perp(x,s)\ :\ x\in\pM,0\leq s\leq s_0\},\\
\mathcal X_M&= \{\exp^\perp(x,s)\ :\ x\in\pM,0\geq s\geq -s_0\}.
\end{align*}
Define $\rho: M \cup \mathcal X_{Q} \to M$ as $\rho(\exp^\perp(x,s))=\exp^\perp(x,-s)$ when $s>0$ (i.e. $\rho$ reflects $\mathcal X_Q$ onto $\mathcal X_M$ with respect to Fermi coordinates) and $\rho = \mathrm{id}$ on $M$. Let $\|d \rho \|(p):= \sup_{T_{p}M \setminus \{0\}} |d_{p}\rho(v)|_{\rho(p)}/ |v|_{p}$ denotes the operator norm of $d_{p}\rho$. It is not difficult to see that  $\|d\rho \|(\exp^\perp(x,s))\to 1$ as $s\to 0$ for every $x\in\pM$, therefore we can choose the function $s_0$ so small, depending on $\epsilon$, that $\partial \mathcal X_{Q}$ is smooth and $\|d\rho \| \leq 1+ \epsilon$ on $P$. This latter bound implies that $\rho$ is a $(1+\epsilon)$-Lipschitz map. This amounts to show that, given a piecewise $C^{1}$-curve $\gamma : [0,a] \to P$, it holds
\begin{equation}\label{length}
L_{M} (\rho \circ \gamma) \leq (1+\epsilon) L_{M} (\gamma).
\end{equation}
To this aim, we note that $\rho$ is locally Lipschitz in  $P$. The only delicate points are those in the bi-collar neighborhood $\mathcal{X}_{M} \cup \mathcal{X}_{Q }$. But, in this set, $\rho$ is locally Lipschitz with respect to the product metric inherited from $\partial M \times [-1 , 1]$ and local Lipschitzianity does not depend on the ground metric. Now, the image $\rho \circ \gamma : [0,a] \to M$ is locally Lipschitz and its length satisfies
\[
L_{M} (\rho \circ \gamma) = \int_{0}^{a} v_{\rho \circ \gamma}(t) dt,
\]
where $v_{\rho\circ \gamma}$ denotes the speed of the curve in the sense of \cite{BBI}. In view of (c) of Lemma \ref{lemma-length}, since
\[
v_{\rho\circ \gamma}(t) \leq \| d \rho\|(\gamma(t)) \cdot   v_{\gamma}(t) \leq (1+\epsilon) v_{\gamma}(t)
\]
on the open and full measure subset of $[0,a]$:
\[
\gamma^{-1}(P\setminus M) \cup \inte \big([0,a] \setminus \gamma^{-1}(P \setminus M) \big)
\]
then, by integration, we deduce the validity \eqref{length}. 

\end{proof}

\begin{proof}[Proof of Lemma \ref{lemma-extension-2}]
First, we claim that given a locally Lipschitz, divergent path $\gamma : [0,1) \to P$ its (locally Lipschitz) projection $\rho \circ \gamma : [0,1) \to M$ is divergent. Indeed, if $K \subset M$ is a compact set, then $\rho^{-1}(K) = K \cup \rho|_{\mathcal X _{Q}}^{-1}(K \cap \mathcal X_{M})$ is compact in $P$. Therefore, there exists $0 \leq T <1$ such that $\gamma(t) \not \in \rho^{-1}(K)$ for every $T \leq t <1$. It follows that $\rho \circ \gamma (t) \not \in K$ for $T \leq t <1$, proving the claim.

Now, by Theorem \ref{th_HR}, $(M,g_{M})$ is divergent paths complete and therefore $L_{g_{M}}( \rho \circ \gamma) = +\infty$. Since $\rho$ is $2$-Lipschitz, we conclude that $L_{g_{P}}(\gamma) = +\infty$, as desired.
\end{proof}

\begin{proof}[Proof of Lemma \ref{lemma-extension-3}]
 Consider an exhaustion of $N$, i.e. a sequence $\{N_j\}_{j=0}^\infty$ of compact manifolds with smooth boundary such that $N_j\Subset N_{j+1}\subset N$ for all $j\geq0$ and $\cup_{j=0}^\infty N_j=N$. In the following, we use the convention $N_j=\emptyset$ whenever $j<0$. Call:
 \begin{itemize}
\item  $N_{j,a}$ any connected component of $(N\setminus \inte P)\cap (\overline{N_{j+1}\setminus N_{j}})$ for $a\in A_j$;
\item $\hat N_{j,b}$ any connected component of $(N\setminus \inte P)\cap(\overline{N_{j+2}\setminus N_{j-1}})$ for $b\in B_j$. Observe that $\#B_j\leq \# A_j<\infty$ for all $j$.
\end{itemize}
Finally, define
\begin{itemize}
\item $\partial_P\hat N_{j,a}=\hat N_{j,a}\cap \partial P$.
\end{itemize}
 
We have the following

\begin{lemma}\label{lem_metric}
There exists a smooth Riemannian metric $g_{N}$ on $N$ such that $(N,g_{N})$ is a Riemannian extension of $(M,g_{M})$ and, for all $j\in\N$, $a\in A_j$ and $b\in B_j$, the following hold:
\begin{enumerate}
\item [(a)] Let $x,y\in N_{j,a}$ with $x\in \partial N_j$ and $y\in \partial N_{j+1}$.  If $\gamma : [0,1] \to N_{j,a}$ is any Lipschitz path connecting $x$ to $y$ then $L_{g_{N}}(\gamma) \geq 1$.
\item [(b)] Let $x,y \in \partial_P\hat N_{j,b}$.
If $\gamma : [0,1] \to \hat N_{j,b}$ is any Lipschitz path connecting $x$ to $y$ then $L_{g_{N}}(\gamma) \geq d_{(P,g_{P})}(x,y)$.
\end{enumerate}
\end{lemma}

\begin{proof}
For the ease of notation, given a subset $C$ of $(N,\tilde g)$ we shall denote by $\tilde d_{C}$ the length metric on $C$ induced by $(N,\tilde g)$, namely,
\[
\tilde d_{C}(c_{1},c_{2}) = \inf L_{\tilde g}(\gamma)
\]
where the infimum is taken over the Lipschitz path in $C$ (if any) connecting $c_{1}$ with $c_{2}$.\smallskip

For any  $j\in\N$ and $a\in A_j$, define
\[
q_1^{j,a}=\inf \tilde d_{N_{j,a}}(x,y),
\]
where the infimum is taken over all the $x,y\in N_{j,a}$ such that $x\in \partial N_j$ and $y\in \partial N_{j+1}$. Since $\partial N_{j+1}\cap (N\setminus \inte P)$ and $\partial N_j\cap (N\setminus \inte P)$ are  compact and disjoint, $q_1^{j,a}>0$.\\
For any $j\in\N$ and $b\in B_j$, define
$$
 \delta^{j,b}(x,y)= \frac {\tilde d_{\hat N_{j,b}}(x,y)}{d_{(P,g_{P})}(x,y)},
$$
and
\[
q_2^{j,b}=\inf  \delta^{j,b}(x,y),
\]
where the infimum is taken over all the $x\neq y$ belonging to $\partial_P\hat N_{j,b}$. We claim that $q_2^{j,b}>0$. Indeed, suppose $q_{2}^{j,b}=0$. Then, there exist sequences of points $\{x_{k}\}$ and $\{y_{k}\}$ in $\partial_{P} \hat N_{j,b} \subset \partial P$  such that $\delta^{j,b}(x_{k},y_{k}) \to 0$. Since $\tilde d_{\hat N_{j,b}} \geq d_{(N,\tilde g)}$ on $\hat N_{j,b}$, we deduce that
\[
\frac { d_{(N,\tilde g)}(x_{k},y_{k})}{d_{(P,g_{P})}(x_{k},y_{k})} \to 0.
\]
Since $x_{k},y_{k}$ are in a compact subset of $P$ then the denominator  $d_{(P,g_P)}(x_{k},y_{k})$ is uniformly bounded. It follows that $d_{(N,\tilde g)}(x_{k},y_{k}) \to 0$. Therefore, by compactness of $\partial_{P} \hat N_{j,b}$, and up to passing to subsequences, we can assume that $\{x_{k}\} , \{y_{k}\}$ converge to a same point $z \in \partial_{P} \hat N_{j,b}$ with respect to the $d_{(N,\tilde g)}$ metric.
Since $P$ is a manifold with smooth boundary,
 \[
 \frac { d_{(N,\tilde g)}(x_{k},y_{k})}{d_{(P,g_{P})}(x_{k},y_{k})}= \frac { d_{(N,\tilde g)}(x_{k},y_{k})}{d_{(P,\tilde g)}(x_{k},y_{k})} \to 1,
 \]
 a contradiction.\\

For every $j\in\N$, $a\in A_j$ and $b\in B_j$, let $\mu_{j,a},\nu_{j,b}\in C^{\infty}_c((N\setminus M)\cap N_{j+2})$ be such that $0\leq\mu_{j,a},\nu_{j,b},\leq 1$, 
$$\mu_{j,a}|_{N_{j,a}}\equiv 1,\ \mu_{j,a}|_{N_{j-1}}\equiv 0,\ \nu_{j,b}|_{\hat N_{j,b}}\equiv 1,\ \nu_{j,b}|_{ N_{j-2}}\equiv 0.
$$
We define the smooth Riemannian metric $g_{N}$ on $N$ as
\begin{align*}
g_{N}(x)=e^{2\sum_{j=0}^\infty\left[\sum_{a\in A_j}\max\{0;-\ln(q_1^{j,a})\}\mu_{j,a}(x)
+\sum_{b\in B_j}\max\{0;-\ln(q_2^{j,b})\}\nu_{j,b}(x)\right]}
\tilde g(x).
\end{align*}
Note that $g_{N}$ is well defined, since the sum is locally finite. Moreover the conformal factor is everywhere greater or equal to $1$, and it is greater or equal to $(q_1^{j,a})^{-2}$ on $N_{j,a}$ and to  $(q_2^{j,b})^{-2}$ on $\hat N_{j,b}$. So the metric $g_{N}$ satisfies the claim of the lemma.
\end{proof}

To conclude the proof of Lemma \ref{lemma-extension-3}, we have to show that the metric $g_{N}$ of $N$ obtained in Lemma \ref{lem_metric} is (divergent paths) complete. To this end, we take a locally Lipschitz divergent path $\gamma :[0,1) \to N$  and we distinguish three different cases:\\

\noindent \textit{First case}. The path $\gamma$ is definitely contained in $N \setminus P$. Without loss of generality we can assume that the entire path $\gamma$ is contained in $N \setminus P$. Using item (a) of Lemma \ref{lem_metric} we easily deduce that $L_{g_{N}}(\gamma) = +\infty$.\\

\noindent \textit{Second case}. The path $\gamma$ is definitely contained in $\inte P$. As above, we can assume that $\gamma$ is entirely in $P$. Then, by assumption, $L_{g_{P}}(\gamma) = +\infty$. On the other hand, by definition of $g_N$ we have that  $L_{g_{N}} \geq L_{g_{P}}$ and, therefore, $L_{g_{N}}(\gamma) = +\infty$.\\

\noindent \textit{Third case}. There exists a sequence of times $t_{k} \to 1^{-}$ such that $\gamma (t_{2k}) \in N \setminus  P$ and $\gamma(t_{2k+1}) \in \inte P$ for all $k$. By contradiction, let us assume that $L_{g_{N}}(\gamma) < +\infty$. Then, up to starting from $T$ close enough to $1$ we can assume that $\ell := L_{g_{N}}(\gamma) < 1$ and that $\gamma(0) \in P$.
We consider the natural reparametrization of $\gamma$ and we assume that $\gamma : [0,\ell) \to N$ has unit speed; \cite[Proposition 2.5.9]{BBI}.

Consider the disjoint union 
\[
\gamma^{-1}(N \setminus P) = \dot \cup_{\lambda \in \nn} (\a_{\l}, \b_{\l}).
\]
Then, by item (a) of Lemma \ref{lem_metric}, for each $\l$ there exist $j_{\l}\in \nn$ and $b_{\l} \in B_{j_{\l}}$ such that
\[
\gamma((\a_{\l},\b_{\l})) \subset \hat N_{j_{\l},b_{\l}}.
\]
By item (b) of Lemma \ref{lem_metric}, for every $\l$,
\[
d_{(P,g_{P})}(\gamma(\a_{\l}), \gamma(\b_{\l})) \leq L_{g_{N}} (\g|_{(\a_{\l},\b_{\l})}).
\]
Hence there exists a Lipschitz curve $\s_{\l} : [\a_{\l},\b_{\l}] \to P$ with the same endpoints of $\gamma|_{[\a_{\l},\b_{\l}]}$, i.e.,
\[
\s_{\l}(\a_{\l}) = \gamma(\a_{\l}),\quad \s_{\l}(\b_{\l}) = \gamma(\b_{\l}),
\]
and such that
\begin{equation}\label{length-g}
L_{g_{P}}(\s_{\l}) \leq 2 L_{g_{N}} (\g|_{(a_{\l},\b_{\l})}) = 2( \b_{\l} - \a_{\l} ).
\end{equation}
We now construct a new path $\sigma : [0,\ell) \to P$ by setting
\[
\sigma(t) = \begin{cases}
\s_{\l}(t) &\text { if } t\in (\a_{\l},\b_{\l}), \text{ for some }\l \in \nn \\
\gamma(t) &\text {otherwise. }
\end{cases}
\]
Set $\mathcal A_n:=\cup_{\l=0}^n(\a_\l,\b_\l)$. For every $n\in\N$ we introduce the $d_{(N,\tilde g)}$-rectifiable paths $\gamma_n:[0,\ell)\to N$ by
\[
\gamma_n(t)=\begin{cases}
\s_{\l}(t) &\text { if } t\in \mathcal A_n,\\
\gamma(t) &\text {otherwise. }
\end{cases}
\]
From \eqref{length-g}, item (b) of Lemma \ref{lemma-length}, and the fact that, by construction, lengths with respect to $\tilde g$ are smaller than lenghts with respect to $g_N$, we deduce that for all $n\in \N$,
\begin{align*}
L_{\tilde g}(\g_n)&=L_{\tilde g}(\g|_{[0,\ell)\setminus \mathcal A_n})+ \sum_{i=0}^n L_{\tilde g}(\s_\l|_{(a_\l,\b_\l)})\\
&= L_{\tilde g}(\g|_{[0,\ell)\setminus \mathcal A_n})+ \sum_{i=0}^n L_{g_P}(\s_\l|_{(a_\l,\b_\l)})\\
&\leq L_{g_N}(\g|_{[0,\ell)\setminus \mathcal A_n})+ \sum_{i=0}^n 2L_{g_N}(\g|_{(a_\l,\b_\l)})\\
&\leq 2L_{g_N}(\g)=2\ell.
\end{align*}
By the semi-continuity of $L_{\tilde g}$ we get that $\s$ is $d_{(N,\tilde g)}$-rectifiable, and 
\begin{equation}\label{finite-length}
L_{g_P}(\s)=L_{\tilde g}(\s)\leq 2\ell.
\end{equation}
Namely, for any fixed $S\in(0,\ell)$ and for any finite partition
$0=s_0<s_1<\dots<s_K=S$, 
there exists $n\in N$ such that $\gamma_n(s_j)=\sigma(s_j)$ for all $j=0,\dots,K$, so that
\[\sum_{j=1}^K d_{(N,\tilde g)}(\s(s_{j-1}),\s(s_{j}))=\sum_{j=1}^K d_{(N,\tilde g)}(\gamma_n(s_{j-1}),\gamma_n(s_{j}))\leq L_{\tilde g}(\g_n)\leq 2\ell.
\]
Finally we show that $\s$ is divergent in $(P,g_P)$. This fact, together with \eqref{finite-length} will contradict the divergent paths completeness of $(P,g_P)$, thus concluding the proof of Lemma \ref{lemma-extension-3}.

To this purpose, fix a compact $C\subset P$ and let $j$ be large enough so that $C\subset N_j$. Since $\g$ is divergent in $N$, there exists $T\in [0,\ell)$ such that $\g(t)\not \in N_{j+1}$ for all $t\in (T,\ell)$. Set 
\[
\mathcal T:=\{\l\in\N\ :\ a_\l>T\text{ and }\s_\l([\a_\l,\b_\l])\cap N_j\neq \emptyset\}.\]
If $\mathcal T$ is empty, there is nothing to prove. Otherwise note that, for every $\l\in\mathcal T$, $\sigma_\l(\a_\l)\not \in N_{j+1}$ and $\sigma_\l(\b_\l)\not \in N_{j+1}$. Define
\[
c_j:=\min\{d_{(N,\tilde g)}(x,y)\ :\ x\in \partial N_j\text{ and }y\in \partial N_{j+1}\},
\]
which is well defined by compactness, and strictly positive since $N_j\Subset N_{j+1}$.
Then 
\[
\sharp\mathcal T \leq  \frac{L_{g_p}(\s)}{
2c_j}<\infty.
\]
Accordingly, we have that $\b^\ast:=\max_{\l\in\mathcal T}b_\l$ satisfies $\b^\ast<\ell$ and $\s([\b^\ast,\ell))\subset N\setminus N_{j}\subset N\setminus C$.
\end{proof}

\section{Some applications}

According to Theorem \ref{theorem-extension}, a Riemannian manifold with smooth boundary can be always realized as a  domain of a Riemannian manifold without boundary. Moreover, the ambient manifold can be chosen to be geodesically complete if the original manifold with boundary was metrically complete (hence a closed domain). This viewpoint on manifolds with boundary has two main consequences: on the one hand, open relations concerning Riemannian quantities of local nature extend trivially past the boundary of the manifold. On the other hand, by restriction, one can easily inherit basic results and constructions from complete manifolds without boundary.  We shall provide examples of both these instances.

\subsection{Local extensions with curvature constraints}
Let $(M,g_{M})$ be a Riemannian manifold with boundary $\partial M \not= \emptyset$ satisfying a strict curvature condition like $\curv_{M}>C$ or $\curv_{M} < C$ for some constant $C \in \rr$. Here, $\curv$ denotes either the sectional, the Ricci or the scalar curvature of the manifold at hand.\smallskip

Consider any Riemannian extension $(N,g_{N})$ of $(M,g_{M})$. Since $\curv >C$ (resp. $\curv < C$) and $\curv_{M} = \curv_{N}$ on $M$, by continuity there exists a neighborhood $U \subseteq N$ of $\partial M$ such that $\curv > C$ (resp. $\curv < C$) holds on $V = M \cup U$. \smallskip

Assume now that $\partial M$ is compact. We say that $\partial M$ is strictly convex (resp. strictly concave) if, with respect to the outward pointing Gauss map $\nu$, the eigenvalues $\lambda_{1},\cdots,\lambda_{m-1}$ of the shape operator $\mathcal{S}(X) = -\text{ }^{N}\!D_{X}\nu$ satisfy $\lambda_{j} <0$ (resp. $>0$). We choose $0 < \delta \ll 1$ in such a way that the normal exponential map $\text{ }^{N}\!\exp^{\perp}: \partial M \times (-\delta , \delta) \to V$
defines e diffeomorphism onto its image  and we can consider the corresponding family of  (diffeomorphic) parallel hypersurfaces
\[
(\partial M)_{t} = \text{ }\!^{N}\!\exp^{\perp}(\partial M \times \{t\}).
\]
Let $\mathcal{S}_{t}$ denote the shape operator of $(\partial M)_{t}$. It is known that its eigenvalues $\lambda_{1}(t),\cdots,\lambda_{m-1}(t)$ evolve (for a.e. $t$) according to the Riccati equation
\[
\frac{d \lambda_{j}}{dt} (t) = \lambda_{j}^{2} (t) + \sect_{N} (\nu \wedge E_{j}(t))
\]
where $E_{j}(t) \in T (\partial M)_{t}$ is the eigenvector of $\mathcal{S}_{t}$ corresponding to $\lambda_{j}(t)$; see e.g. \cite{G}. From this equation, under curvature restrictions and using comparison arguments, one could obtain sign conclusions on suitable intervals. Anyway, regardless of any curvature assumption, if
\[
\lambda_{j}(0) = \lambda_{j} <0
\]
(resp. $>0$), by continuity we find $0<\epsilon < \delta$ such that
\[
\lambda_{j}(t) < 0,\quad 0 \leq t \leq \epsilon,
\]
(resp. $>0$). Clearly, similar considerations hold for the mean curvature function. Thus, by taking $\bar N := M \cup \text{ }\!^{N}\!\exp^{\perp}(\partial M \times [0,\epsilon])$, we have proved the following result.

\begin{corollary}\label{corollary-localcurvature}
 Let $(M,g_{M})$ be a Riemannian manifold with boundary $\partial M \not= \emptyset$ and satisfying $\curv_{M} > C$ (resp. $\curv_{M} <C$). Then, there exists a Riemannian extension $(\bar N,g_{\bar N})$ of $M$ such that $\curv_{\bar N} > C$ (resp. $\curv_{\bar N} <C$). Moreover, assume that $\partial M$ is compact. If $\partial M$ is either strictly (mean) convex  or strictly (mean) concave, then $\bar N$ can be chosen so to have a boundary $\partial \bar N$ with the same property.
\end{corollary}
\begin{remark}
 \rm{
The situation is significantly more difficult if we replace the strict inequalities with their weak counterparts. In this case, smooth local extensions are clearly not allowed in general. The reader may consult \cite{R} for interesting existence results in this setting. 
}
\end{remark}

\subsection{Sobolev spaces}
In the geometric analysis on manifolds with boundary, the theory of (first order) Sobolev spaces, and the corresponding  density results, are vital to carry out PDE's constructions typical of the setting of manifolds without boundary. By way of example, we can mention the truncation method in order to obtain sub(super) solutions of Neumann problems for the Laplace operator and its applications to potential theory; see \cite{IPS}. In this respect, the Euclidean arguments work almost verbatim once we consider the manifold with boundary as a domain inside an ambient manifold without boundary. We are going to illustrate quickly this viewpoint by recovering a classical density result \'a la Meyers-Serrin; see e.g. \cite[Appendix A]{IPS}.\smallskip

Let $(M,g_{M})$ be a (possibly non-compact and incomplete) Riemannian manifold with boundary $\partial M \not = \emptyset$. Since $\inte M$ is a smooth manifold without boundary we can define, as usual, the space
\[
W^{1,p}(\inte M) = \{u:\inte M \to \rr: u \in L^{p}, \nabla u \in L^{p}\},
\]
where $\nabla u$ is the distributional gradient of $u$, endowed with the norm
\[
\| u \|_{W^{1,p}} = ( \| u \|^{p}_{L^{p}} + \| \nabla u \|^{p}_{L^{p}} )^{1/p}.
\]
Suppose now that $(M,g_{M})$ is complete and let $(N,g_{N})$ be a geodesically complete Riemannian extension without boundary. Fix a  locally finite, relatively compact, smooth atlas $\{(V_{j} , \varphi_{j})\}$ of $N$ such that either $V_{j} \cap M= \emptyset$ or $(U_{j} \cap M , \varphi_{j}|_{M})$ is a smooth chart of $M$. Without loss of generality, we can assume that $\phi_{j}(U_{j}) = \mathbb{B}_{1} \subset \rr^{m}$ and (in case $U_{j} \cap \pM \not= \emptyset$) $\phi_{j} (U_{j} \cap M) = \mathbb{B}^{+}_{1}$. We consider a partition of unity $\{\chi_k\}$ subordinated to the covering $\{U_k\}$ and, given a function $u \in W^{1,p}(\inte M)$, we decompose it as  $u= \sum_k u_k$ with $u_k = u \cdot \chi_k$. Now, for any fixed $\epsilon >0$, applying in local coordinates the standard approximation procedure, e.g. \cite[Theorem 10.29]{Le}, we find $\bar u_k \in C_{c}^{\infty}(U_k)$ such that
\[
\| u_k  - \bar u_k\|_{W^{1,2}(\inte M)} \leq \frac{\varepsilon}{2^k}.
\]
Thus, the locally finite sum $\bar u =\sum_k \bar u_k$ is a function in $C^{\infty}(N)$ and gives an $\varepsilon$-approximation of $u$ in the space $W^{1,2}(\mathrm{int}M).$ This implies the partial result:
\begin{equation}\label{MS2}
 W^{1,p}(\inte M) = \overline{C^{\infty}( M) }^{\| \cdot \|_{W^{1,p}}}.
\end{equation}
Finally, we have to approximate $\bar u|_{M}$ in $W^{1,p}(\inte M)$ with the restriction to $M$ of a function in $C^{\infty}_{c}(N)$. To this end take your favorite smooth function $\rho_{N} : N \to \rr_{>0}$ satisfying $\rho_{N} (\infty) = +\infty$ and $\| \nabla \rho_{N} \|_{L^{\infty}(N)} \leq L$. It can be obtained by regularizing the distance function by convolution methods; \cite{GW}. Moreover, choose  $\psi : \rr \to [0,1]$ to be any smooth function such that $\psi(t)  = 1$ if $t \leq 1$ and $\psi(t) = 0$ if  $t \leq 2$, and define the sequence $\psi_{k} := \psi (\rho_{N}/k)\in C^{\infty}_{c}(N)$. Then, $ \psi_{k}  \to 1$,  as $k \to +\infty$ uniformly on compact subsets of  $N$, and $\|\nabla \psi_{k}\|_{L^{\infty}(N)} \to 0$   as $k\to +\infty$.
It is then obvious, by dominated convergence, that  the sequence
\[
\bar {\bar u}_{k} = \bar{\bar u} \cdot \psi_{k}\in C^{\infty}_{c}(N)
\]
converges in $W^{1,p}(N)$ to $\bar{\bar u}$. By restriction, $\bar {\bar u}_{k}|_{M} \in C^{\infty}_{c}(M)$ converges to $\bar{\bar u}|_{M}$ in $W^{1,p}(\inte M)$. We have thus obtained the stronger density result:
\begin{corollary}\label{corollary-density}
Let $(M,g_M)$ be a complete Riemannian manifold with (possibly empty) boundary $\partial M$. Then
\begin{equation}\label{MS3}
  W^{1,p}(\inte M) = \overline{C_{c}^{\infty}( M) }^{\| \cdot \|_{W^{1,p}}}.
\end{equation}
\end{corollary}
As a side product, observe that by taking $\rho_{M} = \rho_{N}|_{M}$ we  also obtain the existence of a smooth, globally Lipschitz, exhaustion function on any complete manifold with boundary.
\begin{lemma}\label{lem_exhaustion}
 Let $(M,g_{M})$ be a complete Riemannian manifold with boundary $\partial M \not= \emptyset$. Then, there exists a smooth function $\rho_{M} : M \to \rr_{>0}$  satisfying
\begin{equation}\label{exhaustion}
\rho_{M} (\infty) = +\infty; \quad \| \nabla \rho_{M} \|_{L^{\infty}(M)} \leq L.
\end{equation}
\end{lemma}
The proof of this fact is not completely obvious if we use the pure viewpoint of manifolds with boundary. The mollification procedure used to regularize a given Lipschitz function (e.g. the intrinsic distance function) requires some care.

\subsection{Proper Nash embedding} The classical formulation of the Nash embedding theorem states that any Riemannian manifold $(N,g_{N})$ can be isometrically embedded in some Euclidean space $\rr^{\ell}$, where $\ell = \ell(\dim N)$. In the ``survey'' part of the paper \cite{GR} it is claimed that the embedding can be chosen to be proper if $N$ is geodesically complete and, moreover, that the Nash embedding holds also for manifolds with boundary. An elementary, but clever, proof of the first claim can be found in \cite{Mul}. Here, we point out that the second claim can be trivially deduced from the first one,  by restricting to the manifold with boundary a proper isometric embedding of a complete Riemannian extension.
\begin{corollary}\label{corollary-nash}
Let $(M,g_{M})$ be a complete Riemannian manifold with boundary $\partial M \not= \emptyset$. Then, there exists a proper isometric embedding of $M$ into some Euclidean space  $\rr^{\ell}$ where $\ell = \ell(\dim M)$.
\end{corollary}

\end{document}